\documentclass[ 11pt,a4paper]{amsart}

\usepackage{amsfonts}
\usepackage{amsmath}
\usepackage{amsmath,amsthm, amssymb}
\usepackage{amscd}
\usepackage{verbatim}
\usepackage{graphicx}
\usepackage{comment}
\usepackage{extarrows}
\usepackage[colorinlistoftodos,textsize=small]{todonotes}
\usepackage[font=small]{caption}

\DeclareRobustCommand{\SkipTocEntry}[4]{}

\newtheorem{thm}{Theorem}[section]

\newtheorem{cor}[thm]{Corollary}
\newtheorem{lem}[thm]{Lemma}

\newtheorem*{thmn}{Theorem \ref{thm:homeo}}

\theoremstyle{definition}
\newtheorem{defi}[thm]{Definition}
\newtheorem{ex}[thm]{Example}

\newtheorem{rem}[thm]{Remark}

\newtheorem{ques}[thm]{Question}

\numberwithin{equation}{section}

\def\RR{\mathbb R}

\def\NN{\mathbb N}

\def\Fbot{V(F)^\bot}
\def\Fo{F^\circ}

\DeclareMathOperator{\ri}{ri}
\DeclareMathOperator{\aff}{aff}
\DeclareMathOperator{\cl}{cl}
\DeclareMathOperator{\conv}{conv}
\DeclareMathOperator{\inte}{int}
\DeclareMathOperator{\dimm}{dim}

\def\delrel{\del_{\text{rel}}}
\def\Bo{B^{\circ}}
\def\Co{C^\circ}
\def\ra{\rightarrow}

\def\lora{\longrightarrow}
\def\coloneqq{\mathrel{\mathop:}=}
\def\del{\partial}
\def\normm{\lVert \cdot \rVert}

\def\xhor{\overline{X}^{hor}}
\def\delhor{\del_{\text{hor}}}

\newcommand\norm[1]{
	\lVert #1 \rVert 
}

\newcommand\fracnorm[1]{
	\frac{#1}{\norm{#1}}
}
\newcommand\pair[2]{
	\left \langle #1 | #2 \right \rangle
}

\newcommand\pairl[2]{
	\left \langle \left. #1 \right | #2 \right \rangle
}

\newcommand\pairr[2]{
	\left \langle #1 \left| #2 \right. \right \rangle
}

\begin{document}

\title[Polyhedral Horofunction Compactification]{Polyhedral Horofunction Compactification as A Polyhedral Ball}

\author[L. Ji]{Lizhen Ji}
\address{Lizhen Ji\\Department of Mathematics, University of Michigan, USA}

\author[A. Schilling]{Anna-Sofie Schilling}
\address{Anna-Sofie Schilling\\Department of Mathematics, University of Heidelberg, Germany}

\date{\today}

\begin{abstract}
In this paper we answer positively a question raised by Kapovich and Leeb in a paper titled ``Finsler bordifications of symmetric and certain locally symmetric spaces". Specifically, we show that  for a finite-dimensional vector space with a polyhedral norm, its horofunction compactification is homeomorphic to the dual unit ball of the norm by an explicit map. To prove this, we  establish  a criterion for converging sequences in the horofunction compactification and generalize the basic notion of the moment map in the theory of toric varieties. 
\end{abstract}
 
 \maketitle
 
\tableofcontents


\section{Introduction}
The horofunction compactification was introduced by Gromov \cite[\S 1.2]{gr} in 1981 and can be constructed for any proper metric space. Due to its generality, this method of compactifying is used in many areas: Rieffel \cite{ri} was interested in a specific subset of the horofunctions, the so-called Busemann points, in order to describe the topology of group \(C^*\)-algebras with Dirac operators coming from length functions. More geometrically, Alessandrini, Liu, Papadopoulos and Su \cite{alessandrini} determined the horofunction compactification with respect to the arc metric for Teichmüller spaces of surfaces with boundary. Several works deal with the horofunction compactification of symmetric spaces of non-compact type and compare it to other well-known compactifications of these spaces, see \cite{hsw}, \cite{kl} and \cite{lemmens}. Many of these works are based on a consideration of the compactification of finite-dimensional vector spaces, whereby they differ in their methods. In \cite{cks}, Ciobotaru, Kramer and Schwer  determine the horofunction compactification of finite-dimensional vector spaces with asymmetric metric in order to study horofunctions on Euclidean buildings in a follow-up paper. Hence it is appropriate to take a closer look at the horofunction compactification of finite-dimensional vector spaces. For this setting, Walsh \cite{wa2} explicitly described the set of Busemann points, the specific subset of horofunctions that was already considered by Rieffel, and gave a criterion to determine whether all horofunctions are Busemann points. 

In this paper we focus on the case of a finite-dimensional normed vector space \(X\) with a polyhedral norm, meaning that the unit ball of the norm is a convex polytope containing the origin in its interior. For the same setting, Karlsson, Metz and Noskov \cite{KMN} use the blow-up-and-shift-technique to characterize the horofunction compactification, whereas Ciobotaru, Kramer and Schwer \cite{cks} use ultrapowers. An advantage of polyhedral norms is that in this case all horofunctions are Busemann points, and that, by Walsh's result in \cite{wa2}, the set of  horofunctions can be determined explicitly. To see the underlying topology of the set of horofunctions, the second author of this paper used Walsh's description to characterize converging sequences in the horofunction compactification of finite-dimensional vector spaces with polyhedral norms in her diploma thesis \cite{sch14}. This result is also given as Theorem \ref{thm:characterization} in this paper and shows a deep connection between the horofunction compactification and the shape of the dual unit ball, a convex polytope in the dual space \(X^*\). This connection was also noticed by Kapovich and Leeb \cite{kl}, who posed the following question: 

\begin{ques}\cite[Quest. 6.18]{kl}
    Let \((X, \normm)\) be a finite-dimensional real vector space with polyhedral norm. Is it true that the horofunction compactification $\xhor$ corresponding to this norm is homeomorphic to the closed dual unit ball and respects the face structure?
    \footnote{The original question states: ''Suppose that $\normm$ is a polyhedral norm on a finite-dimensional real vector space $V$. Is it true that the horoclosure $\overline{V}$ of $V$ with respect to this norm, with its natural stratification,  is homeomorphic to the closed unit ball for the dual norm? ``}
\end{ques}

The main purpose of this paper is to give a positive answer to this question and to give an explicit formula for the homeomorphism, see Theorem \ref{thm:homeo} below. \\

Let us give some details on the construction of the homeomorphism. We identify our finite-dimensional normed vector space $X$ with $\RR^m$ in order to use the Euclidean inner product to define orthogonal projections in $X$. Following  Walsh \cite[Thm. 1.1 and 1.2]{wa2}, we describe the horoboundary as a set of real-valued functions $h_{E,p}$ on the dual space $X^*$ that depend on proper faces $E$ of the dual unit ball $\Bo$ and  points $p$ in a certain subspace of $X$. More details on the functions \(h_{E,p}\), including an explicit description, can be found in Section \ref{sec:h_ep}. These maps $h_{E,p}$ are then used to define the homeomorphism between the horofunction compactification $\overline{X}^{hor}$ of $X$ and the dual unit ball $\Bo$. The map for this homeomorphism was inspired by the moment map from the theory of toric varieties. See \cite[p. 82]{fu} for a definition of the moment map and a similar statement to ours about the map.  The basic construction is as follows: given an $m$-dimensional convex polytope $C$ with vertices $\{c_1, \ldots, c_r\}$ in the dual vector space \(X^*\), we define a bijective map $m^C$ from the  vector space $X$ into the interior of \(C\) by  
\begin{align*}
      m^C: X &\lora \inte(C),\\
      x &\longmapsto \sum_{i = 1}^r \frac{e^{-\langle c_i  | x \rangle}}{\sum_{k =1 }^r e^{-\langle c_k  | x \rangle}} c_i.
\end{align*}

Each exponent in this expression consists of the negative dual pairing between a vertex \(c_i\) of the polytope in \(X^*\) and the point \(x \in X\). The fractions are positive and add up to one. Thus they form the weights for a convex combination of the vertices of \(C\). Depending on the position of \(x\) in relation to \(c_i\), the weighting by \(e^{-\langle c_i  | x \rangle}\) is larger or smaller and the image of \(x\) is "pulled" more or less strongly in the direction of \(c_i\). Therefore this map is open and bijective onto the interior of $C$. More about it can be found in Section \ref{sec:m^c}.

A combination of several maps of the form $m^C$, where now $C \subseteq \Bo$ is a face of the dual unit ball $\Bo$, is used to define the homeomorphism \(m\) between the compactification \(\xhor\) and the dual unit ball \(\Bo\): 

\begin{thmn} 
	Let $(X, \normm)$ be a finite-dimensional normed space with polyhedral norm. Let $B \subset X$ be the unit ball associated to $\normm$ and $\Bo \subset X^*$ its dual polytope. Then the horofunction compactification $\overline{X}^{hor}$ of $X$ with respect to the norm $\normm$ is homeomorphic to $\Bo$ via the map
	\begin{align*}
		m: \xhor \lora \Bo, \hspace{0.8cm} \begin{cases} \hfill	X \ni x  &\longmapsto m^{\Bo}(x), \\
		\delhor X \ni h_{E,p}   &\longmapsto m^E(p). \end{cases} 
	\end{align*}
\end{thmn}

Let us shortly explain what this map does. The space $X$ is mapped into the interior of $\Bo$ by the map \(m^{\Bo}: X \ra \inte(\Bo)\). Horofunctions \(h_{E,p}\) associated to the face $E \subset \Bo$  from the boundary of \(\xhor\) are mapped into  the interior of $E$ by the map \(m^E\) applied to the point \(p \in X\). 
This special structure of the maps \(m^{\Bo}\) and \(m^E\) (with \(E \subset \del\Bo\) extreme) guarantees compatibility with the topology and thus the continuity of the map.

The proof of Theorem \ref{thm:homeo} is based on the above mentioned characterization of sequences converging to the horoboundary given in Theorem \ref{thm:characterization}.
The characterization given in this paper focuses on the case where the norm is polyhedral. A more general version of this result, including non-polyhedral norms with smooth boundaries and general norms for two-dimensional cases, can be found in \cite{sch21}.

 \subsubsection*{Outline of the paper}
After some preliminaries in Section \ref{sec:preliminaries} we prove Theorem \ref{thm:characterization}, the characterization of converging sequences, in Section \ref{sec:horofunction}.  To visualize the strong dependence on the direction and shape of the sequence from the faces of the unit ball and its dual we give some illustrative examples in Section \ref{sec:examples}. By combining this characterization and the above explicit map, we prove Theorem \ref{thm:homeo} in the last section.

\subsubsection*{Acknowledgment}

Both authors would like to express their gratitude to the referees of the previous version of this paper and to Pat Boland for their careful reading and constructive suggestions. The first author acknowledges support from NSF grants DMS 1107452, 1107263, 1107367 GEometric structures And Representation varieties (the GEAR Network) and partial support from Simons Fellowship (grant \#305526) and the Simons grant \#353785. The second author was supported by the European Research Council under ERC-consolidator grant 614733. 


\section{Preliminaries} \label{sec:preliminaries}


\subsection{Notations}

In the following, $(X,\normm)$ always denotes an $m$-dimensional normed real vector space where we allow the norm to be asymmetric, that is, \(\lVert -x \rVert \neq \lVert x \rVert \) for \(x \in X\) is possible. Additionally, we require the norm to be polyhedral, that is, the unit ball $B$ associated to the norm $\normm$ is an $m$-dimensional polytope containing the origin in its interior.  
Let $\langle \cdot | \cdot \rangle$ denote the dual pairing of the dual space $X^*$ and $X$. 
 
\subsection{Cones, subspaces and convex hulls}

We remind you of basic definitions and standardize the notation. For visualization see the sketch in Figure \ref{fig:cones_subspaces_notations}.

For any subset $F \subset X$ let $V(F) \subset X$ be the linear subspace generated by $F$, that is, the smallest linear subspace of \(X\)  containing $F$, and let $\Fbot$ denote its orthogonal complement with respect to the Euclidean inner product obtained by identifying $X$ with $\RR^m$. The orthogonal projection of $x \in X$ to these two subspaces will be denoted by $x_F$ for the projection to $V(F)$ and $x^F$ for the projection to $\Fbot$. 
The affine hull of a set $A \subset X$ is the smallest affine subspace in $X$ containing $A$ and it is denoted by $\aff(A)$.  With it we define the dimension of the set \(A\) to be \(\dimm(A) \coloneqq \dimm(\aff(A))\). 

The cone \(K_F\) over $F$ is the set 
\[
	K_F \coloneqq \{ x \in X | \exists \alpha \geq 0, f \in F \text{ such that } x = \alpha f\}.
\] 
Note that all cones \(K_F\) are pointed, that is, they contain the origin \(\{0\}\). Let \(H\) be an affine hyperplane intersecting the cone \(K_F\) but not containing the origin. Then the subset \(K_F \cap V_H^c\) is called a \emph{truncated cone}, where \(V_H^c\) denotes the half-space defined by \(H\) which does not contain the origin (the upper index "c" stands for "complement"). 

\begin{figure}[h!]
	\centering
	\includegraphics[scale=1]{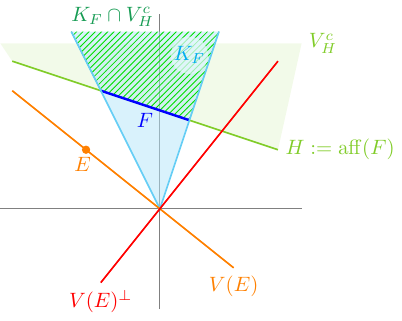}
	\caption{Schematic picture to illustrate the notations} 	\label{fig:cones_subspaces_notations}
\end{figure}

Given a Set \(S \subset X\), the convex hull over \(S\) can be described as the set of all convex combinations with elements from \(S\): 
\[
	\conv(S) \coloneqq \left\{\left. \sum_{i \in \NN}^{k} \alpha_i s_i \ \right|\  k\in \NN, s_i \in S, \alpha_i \geq 0, \sum_{i = 1}^k \alpha_i = 1 \right\}.
\]
Whenever a compact convex subset $C$ is given as the convex hull of over finitely many points, $C = \conv\{c_1, \ldots, c_k\}$, we will assume this subset of points to be minimal, that is, $\conv\{c_1, \ldots, c_k\} \neq \conv\{c_1, \ldots, c_{j-1}, c_{j+1}, \ldots, c_k\}$ for all $j = 1, \ldots, k$. This means that each point $c_j$ is a proper vertex  of $C$. 

\begin{rem}
	We could also have taken the quotient $X/V(F)$ instead of $\Fbot$, but since the orthogonal complement is more geometric, we use the complement $V(F)^\bot$.
\end{rem}

\subsection{Some convex analysis}

\begin{defi}\label{defi:dual}
	Let $B$ be the unit ball of our norm $\normm$. Then the \emph{dual unit ball} $\Bo$  is defined as the polar of $B$:
	\[
		\Bo \coloneqq \{ y \in X^* \ | \langle y | x \rangle \geq -1 \ \forall x \in B \}. 
	\]
	The \emph{dual norm} $\normm^\circ$ is the norm which has $\Bo$ as its unit ball. 
\end{defi}

Some authors define the dual unit ball by the condition \(\pair{y}{x} \leq 1 \) with \(y \in X^*\) and \(x \in B\). Both definitions are valid but result in different pictures and a careful treatment of signs in the calculations is necessary. For more details on polars and polyhedral convex sets see for example \cite{be} or \cite[\S 19]{ro}.

\begin{rem}
	Every $m$-dimensional polytope $C \subset X$ containing the origin in its interior defines a norm $\normm_C$ on $X$ via
	\begin{align*}
		\lVert x \rVert_C \coloneqq \inf\{\alpha > 0 | x \in \alpha C\}
	\end{align*}
	for all $x \in X$. 
\end{rem}

There are two ways to describe a bounded convex polytope, either as the convex hull of a finite set of points or as the intersection of finitely many half-spaces. This leads to the following duality in the description of a polytope $C$, which contains the origin in its interior, and of its dual polytope $C^\circ$: 

Let $C = \conv\{c_1, \ldots, c_r\}$ be an $m$-dimensional polytope whose interior contains the origin and which is given as the convex hull of a finite set of points. Then each point $c_i$ (\(i \in \{1, \ldots, r\}\)) defines a hyperplane $H_i \subset X^*$  such that $\langle H_i | c_i \rangle = -1$, that is, $\langle h_i | c_i \rangle = -1$ for all $h_i \in H_i$. Let $V_i  \coloneqq V_{H_i} \subset X^*$ be the closed half-space bounded by $H_i$ which contains the origin. Then \(\pair{v_i}{c_i} \geq -1\) for all \(v_i \in V_i\). Thus
\begin{align*}
	C^\circ &= \bigcap_{i = 1}^r V_i\\
	&= \{ y \in X^* \ | \ \langle y | c_i \rangle \geq -1 \ \forall i= 1, \ldots, r\}.
\end{align*} 
  
As $C$ is convex and contains the origin we have $(C^\circ)^\circ = C$. Consequently it is also easy to describe $C^\circ$ as the convex hull of a finite set of points when $C$ is given as the intersection of certain half-spaces $V_i$.  

\begin{defi}
	A \emph{$k$-face} of a polytope $C = \bigcap_{i = 1}^r V_i \subset X$ is a $k$-dimensional non-empty subset of $X$ which is the intersection of $C$ with one or more hyperplanes $H_i$ that bound $V_i$.  An $(m-1)$-dimensional face is also called a \emph{facet}. 
\end{defi}

Note that all faces of a polytope \(C = \conv\{c_1, \ldots, c_r\} = \bigcap_{i = 1}^r V_i\) are closed sets. According to the previous discussion, such a face can not only be described as the non-empty intersection of hyperplanes, but also as the convex hull of a certain subset of the vertices \(\{c_1, \ldots, c_r\}\) of \(C\). 

\begin{ex} \label{ex:dualL1}
	We consider $\RR^2$ equipped with the $L^1$-norm. Then $B$ is the square with vertices 
	\[
		c_1 = (1,0), \ c_2 = (0,1), \ c_3 = (-1,0), \ c_4 = (0,-1).
	\]
	Each vertex defines a hyperplane \(H_i\) (for \(i \in \{1, \ldots, 4\}\)) by the condition \(\pair{H_i}{c_i} = -1\). We have, for example,
	\[
		H_1 = \left\{ \left . (-1, y) \right| \ y \in \RR \right\},
	\] 
	and similar expressions for the other three hyperplanes. 
	
	\begin{figure}[h!]
		\includegraphics{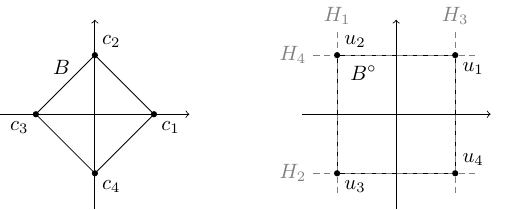}
		\caption{The unit ball $B$ and its dual $\Bo$ as in Example \ref{ex:dualL1} }\label{fig: example_B_Bo}
	\end{figure}
	
	Looking at Figure \ref{fig: example_B_Bo}, we see that the dual unit ball is a square corresponding to the $L^\infty$-norm with vertices 
	\[
		u_1 = (1,1), \ u_2 = (-1,1), \ u_3 = (-1, -1), \ u_4 = (1,-1).
	\]
\end{ex}

\begin{rem}[\cite{hsw}, Lemma 3.8]
	There is a one-to-one correspondence between the faces of $B$ and those of $\Bo$. Indeed, let $F \subset B$ be a $k$-face of the \(m\)-dimensional polyhedral unit ball $B$. Then there is a unique ($m-1-k$)-face $E \subset \Bo$ of the dual unit ball defined by the equation $\langle E | F \rangle = -1$, that is, $\langle e | f \rangle = -1$ for all $e \in E$ and $f \in F$. This face is called the \emph{dual face} of $F$ and often denoted by $F^\circ$. As \(B\) and \(\Bo\) are polyhedral unit balls, we have \((F^\circ)^\circ\). Note that 
	\begin{align*} \label{dimensionformula}
		\dim F + \dim F^\circ = m-1.
	\end{align*}
\end{rem}

One feature of the relation of a face \(F\) and its dual \(E\) is the property that all points in \(E\) have the same dual pairing with a given point in \(V(F)\). The precise statement is as follows: 

\begin{lem}\label{lem:samepairing}
	Let $E \subset \Bo \subset X^*$ be a face and $F \subset B \subset X$ its dual. Then there is a point $t \in V(F)^* \subseteq X^*$ such that 
	\[
		\langle e | q \rangle = \langle t | q \rangle
	\]
	for all $e \in E $ and $q \in V(F)$. 
\end{lem}

\begin{proof}
	The statement follows from the fact\footnote{For a reference see for example the proof of \cite[Lemma 3.8]{hsw}.} that for every face $E \subset \Bo$ there is a point $t \in V(F)^*$ such that $E$ lies in the affine subspace $\left(\Fbot\right)^* + t$, in other words, \(\aff(E) = \left(V(F)^\bot\right)^* + t\). That is, for all $e \in E$ there is an $f^\bot \in \left(\Fbot\right)^*$ such that $e = f^\bot + t$.
\end{proof}

\begin{defi}
	 The \emph{relative interior} $\ri(A)$ of $A$ is the interior of $A$ within $\aff(A)$. 
	 
	 Similarly, we define the \emph{relative boundary} of $A$ as $\del_{rel}A \coloneqq (\cl A)\setminus (\ri A)$. 
\end{defi}

\begin{lem} \label{lem:pairinginfty}
	Let $C \subset X$ be an $m$-dimensional convex polytope with the origin in its interior and denote by \(C^\circ = \conv\{c_1, \ldots, c_r\} \subset X^*\) its dual polytope. Let $(x_n)_{n \in \NN} \subset X$ be a sequence satisfying the following two properties with respect to some face $F \subset C$ and its dual face $E = F^\circ \subset C^\circ$: 
	\begin{itemize}
		\item The sequence \((x_n)_{n \in \NN}\) lies in \(V(F)\) and diverges: $\lVert x_{n} \rVert_C \ra \infty$.
		\item For $n$ large enough, the sequence $(x_{n})_{n \in \NN}$ lies in the cone $K_F$ and is bounded away from its boundary: \mbox{\(d(x_{n}, \delrel K_F) \ra \infty\)}.
	\end{itemize}
	Then for any pair of vertices $c_E \in E$ and \(c_j \in C^\circ\) the following holds, as $n \lora \infty$:
	\begin{align*}
		\langle c_E - c_j | x_{n} \rangle \lora  
		\left
		\{
			\begin{array}{ll}
				0 & \text{ if } c_j \in E, \\
				-\infty & \text{ if } c_j \notin E .
			\end{array}
		\right.
	\end{align*}
\end{lem}

\begin{proof}
	As $x_{n} \in K_F$ for $n$ large enough, there is an $f_n \in F$ for each large $n$ such that $x_{n} = \lVert x_{n} \rVert_C \cdot f_n$. 
	
	Let both $c_E$ and $c_j$ be vertices of $E$. As $f_n \in F = E^\circ$:
	\[
		\langle c_E - c_j | x_{n} \rangle = \lVert x_{n} \rVert_C \Big ( \langle c_E - c_j | f_n \rangle \Big )  = 0.
	\]
	
	If $c_j \notin E$, then $\langle c_j | f_n \rangle > -1$ while $\langle c_E | f_n \rangle = -1$, hence
	\[
		\langle c_E - c_j | x_{n} \rangle = \lVert x_{n} \rVert_C \Big ( \underbrace{\langle c_E - c_j | f_n \rangle }_{< 0} \Big ) < 0. 
	\] 
	
	If \(\pair{c_j}{f_n}\) is bounded away from \(-1\), then \(\pair{c_E - c_j}{x_{n}} < \gamma < 0\) for some \(\gamma > 0\) and we are done, since \(lVert x_n \rVert \ra \infty\). 
	
	Otherwise, we have \(\pair{c_j}{f_n} \lora -1\). This behavior only occurs if the sequence \((f_n)_{n \in \NN} = \left(\fracnorm{x_{n}}\right)_{n\in\NN}\) converges to a point within a face \(F_0\) in the relative boundary of \(F\), and in addition if \(c_j \in  E_0 \coloneqq \Fo_0\) is a vertex of \(E_0\) but not of \(E\). Note that, by the duality of \(C\) and \(\Co\),  \(E \subsetneq E_0\) is a proper extreme set in the relative boundary of \(E_0\). As \(c_E\) and \(c_j\) are both vertices of \(E_0\), the dual pairing simplifies to 
	\[
        \pair{c_E - c_j}{x_n} = \pairr{c_E - c_j}{x_n^{F_0}},
	\]
	since \(E_0\) is parallel to the subspace \((V(F_0)^\bot)^* \subset X^*\). As \(f_n\) converges to a point in \(F_0\) and the distance \(d(x_n, \delrel K_F)\) is unbounded, the projected sequence \(\left(x_n^{F_0}\right)_n\) also diverges because \(K_{F_0}\) is part of the relative boundary of \(K_F\).  
	
	The idea of the rest of the proof is to project the sequence to the subspaces \(V(F_0)^\bot\) and \((V(F_0)^\bot)^*\), respectively, and to take a shift of \(E_0\) as a new dual unit ball. We will then have a similar setting as before but in lower dimension. Doing this projecting procedure possibly several times, we will finally get the desired convergence. 
	
	Let us now discuss the details of the proof, see also Figure \ref{fig:lem_pairinginfty} for a sketch and to remember notations.

	\begin{figure}[h!]
			\centering
			\includegraphics[scale=0.9]{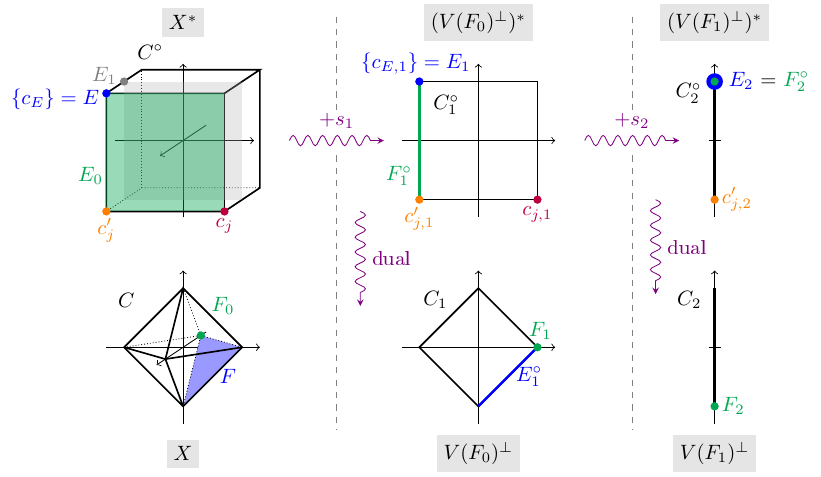}
			\caption{Sketch of the idea of constructing new unit balls.} \label{fig:lem_pairinginfty}
	\end{figure}

	Let \(s_1 \in  X^*\) be a shifting parameter such that 
	\[
        C_1^\circ \coloneqq E_0 + s_1 \subset (V(F_0)^\bot)^*
	\]
    contains the origin in its relative interior. Then \(C_1^\circ\) is a parallel translate of \(E_0\) and has 
    \[
        E_1 \coloneqq E + s_1 \subset (V(F_0)^\bot)^*
    \]
    in its relative boundary. We now only consider the subspaces \((V(F_0)^\bot)^*\) and \(V(F_0)^\bot\). As \(C_1^\circ\) is a compact convex set with the origin in its interior in \((V(F_0)^\bot)^*\),  we now take its dual \(C_1 \coloneqq (C_1^\circ)^\circ \subset X\) in \(V(F_0)^\bot\) as the new unit ball. Thus, we have the following setting now: in the subspaces \(V(F_0)^\bot\) and \((V(F_0)^\bot)^*\) we have the compact convex sets \(C_1\) and \(C_1^\circ\) as the unit and dual unit ball, respectively, and an diverging sequence \(\left(x_n^{F_0}\right)_n \subset V(F_0)^\bot\). The faces of \(C_1^\circ \subset X\) are in one-to-one correspondence with those faces of \(C\) that have \(F_0\) in their relative boundary. The difference of dimensions between corresponding faces of \(C\) and \(C_1\) is \(\dimm(F_0) + 1\). Facets of \(C\) correspond to facets of \(C_1\) and \(\dimm(F_0) + 1\)-dimensional faces of \(C\) become vertices in \(C_1\). Due to this correspondence of faces, the projected sequence \(\left(x_n^{\Fo}\right)_{n \in \NN}\) lives in the cone \(K_{E_1^\circ}\) over the face \(E_1^\circ\) that is dual to the face \(E_1 \subset C_1^\circ\). Note that we consider duality now only within the subspaces \(V(F_0)^\bot\) and \((V(F_0)^\bot)^*\). Additionally, the projected sequence \(\left(x_n^{\Fo}\right)_n\) has unbounded distance to the relative boundary of \(K_{E_1^\circ}\), because the latter consists of projections of faces in the relative boundary of \(K_F\). 
    
    Therefore, we now have the same situation as at the beginning of the proof, but this time with faces \(E_1^\circ\) and \(E_1\) instead of \(F\) and \(E\), and everything happens in a subspace of lower dimension. 
    
    Let \(F_1 \subset C_1\) be the (not necessarily proper) face of \(E_1^\circ\) containing all accumulation points of the sequence \(\left(\fracnorm{x_n^{F_0}}\right)_n\). Denote by  \(c_{E,1} \coloneqq c_E + s_1 \in E_1\) and \(c_{j,1} \coloneqq c_j + s_1 \in C_1^\circ \setminus E_1\) the shifted vertices. If \(c_{j,1} \notin \delrel F_1^\circ \subset C_1^\circ\) (as \(c_{j,1}\) in the middle picture of \ref{fig:lem_pairinginfty}), we know by the first part of the proof that 
    \[
        \pairl{c_E - c_j}{x_n} = \pairr{c_{E,1} - c_{j,1}}{x_n^{F_0}} \lora -\infty.
    \]
    
    If \(c_{j,1} \in \delrel F_1^\circ\) (as \(c_{j,1}'\) in the picture), then we go on again in the same way by projecting (now to \(V(F_1)^\bot\)) and constructing a new unit ball and so on. We continue this procedure as long as \(c_{j,k} \in \delrel F_k^\circ\), with \(k\) denoting the step of iteration. As we lower the dimension of the subspaces in each step, we finally come to the point where the unit ball is one-dimensional with exactly two faces. Then the points \(c_{E,k}, c_{j,k}\) (obtained by several shiftings of \(c_E\) and \(c_j\), respectively), can not be in the relative boundary of the same face any more (because we shift with the same parameter in each step) and thereby the iterating process finally ends. In the last step we therefore have 
    \[
        \pairl{c_E - c_j}{x_n} = \pairr{c_{E,k} - c_{j,k}}{x_n^{F_{k-1}}} \lora -\infty. \qedhere
    \]
\end{proof}

\subsection{The ``pseudo-norm'' $\lvert \cdot \rvert_R$}

In order to be able to introduce the maps defining horofunctions in the next section, we first have to define a so-called ``pseudo-norm'' on \(X\), for more details see \cite[p.5]{wa2}:

\begin{defi}
	Let $R \subset X^*$ be a convex set. For every $x \in X$ define 
	\begin{align*}\label{pseudonorm}
		|x|_R \coloneqq - \inf_{q \in R}\langle q|x\rangle.
	\end{align*}
\end{defi}

Whenever \(R\) is bounded and contains the origin in its relative interior, \(\lvert \cdot \rvert_R\) actually defines a (non-symmetric) norm on \(\aff(R)\) with unit ball \(R^\circ\). Thus, by the duality of the unit balls $B$ and $\Bo$, $\lvert \cdot \rvert_{\Bo}$ is a norm: 
\[
	\lvert \cdot \rvert_{\Bo} = -\inf_{q \in \Bo}\langle q|\cdot\rangle = \lVert \cdot\rVert.
\]  

In the following, we state some technical lemmata about the relation of the pseudo-norm $\lvert \cdot \rvert_R$ and the norm $\normm$. We will use them later in the proof of Theorem \ref{thm:characterization}. 

\begin{lem} \label{lem:minionbdry}
	If $R = \conv\{r_1, \ldots, r_k\}$  is a convex polytope, then 
	\begin{align*}
		\lvert x \rvert_R = - \inf_{i = 1,\ldots, k} \langle r_i | x \rangle.
	\end{align*}
\end{lem}

\begin{proof}
	For any $x \in X$, define a function $f_x: R \lora \RR$ via $f_x(q) = \langle q|x\rangle $. As $R$ is compact and $f_x$ is continuous and affine, $f_x$ takes its minimum and maximum on the boundary of $R$. Indeed, if the extrema would only lie in the interior of $R$, the derivative would be $0$ there. As $f_x$ is affine, it had then to be constant in contradiction to the assumption that it takes its extrema not on the boundary. As the boundary of $R$ is the finite union of several polyhedral convex sets, we can conclude in the same way that $f_x$ takes its minimum and maximum on the vertices $r_1, \ldots, r_k$.
\end{proof}

\begin{lem} \label{lem:||E||=Ei}
	Let $F$ be a proper face of $B$ and $E \coloneqq F^\circ$ its dual face with vertices $\{e_1, \ldots, e_k\}$. Their dual facets are denoted by $F_i \coloneqq \{e_i\}^\circ \subset B$ for all $i = 1, \ldots, k$. If $F \subset B$ is not a facet, all the facets $F_i$ contain $F$ in their relative boundary. If $F$ is a facet, then \(E = \{e_1\}\) is a point (\(k=1\)) and $F = F_1$. Let $x \in K_{F_j}$ be a point for some (not necessarily unique) $j \in \{1, \ldots, k\}$ . Then
	\[
		|x|_E = \norm{x}.
	\]
\end{lem}

\begin{proof}
	Because of the duality  $F_j = \{e_j\}^\circ$ and as $\fracnorm{x} \in F_j$, we know that 
	\[
		\left \langle e_j \left |\fracnorm{x}\right. \right \rangle  = -1, \hspace{1cm}	  \left \langle e_i \left |\fracnorm{x} \right. \right \rangle  \geq -1 \hspace{3mm} \text{ for all } i \neq j.
	\]  
	Using Lemma \ref{lem:minionbdry}, we compute
	\begin{align*}
		|x|_E &= -\inf_{q \in E}\langle q|x\rangle  \\
		&= - \inf_{i = 1, \ldots, k}\langle e_i|x\rangle \\
		&= - \langle e_j|x\rangle \\
		&= \norm{x}. \qedhere
	\end{align*}
\end{proof}

\begin{lem} \label{lem:||E=||B+||E}
      Let $F$ be a proper face of $B$ and $E = F^\circ$ its dual face. Let \(x \in K_F\) be a point. Then, for all $p \in X$, 
      \[ 
	  	|x + p|_E = \norm{x} + |p|_E.
      \]
\end{lem}

\begin{proof}
      As $\frac{x}{\lVert x \rVert} \in F$, we know that  $\langle q|x\rangle  = - \lVert x \rVert$ for all $ q \in E$. With this we obtain
      \begin{align*}
	  |x + p|_E &= -\inf_{q \in E}\langle q|x+p\rangle \\ 
	  &= - \inf_{q \in E}[\langle q|x\rangle  + \langle q|p\rangle ]\\
	  &= \lVert x \rVert - \inf_{q \in E}\langle q|p\rangle \\
	  &= \lVert x \rVert + |p|_E. \qedhere
      \end{align*}
\end{proof}

\subsection{The maps $h_{E,p}$} \label{sec:h_ep}

We are now prepared to introduce real-valued functions \(h_{E,p}\)on $X$, which will later turn out to be the horofunctions of $X$ with respect to our norm $\normm$. For every proper face $E \subset \Bo$ and every point $p \in V(E^\circ)^\bot$ we define the function 
\begin{align*}
      h_{E,p}: X &\lora \RR, \\
      y &\longmapsto \lvert p - y \rvert_E - \lvert p \rvert_E. 
\end{align*}

We could also take $p \in X$ to define $h_{E,p}$, but the following lemma shows that this restriction gives us uniqueness since only the part in $V(E^\circ)^\bot$ contributes to $h_{E,p}$:

\begin{lem} \label{lem:p^Fonly}
      Let $E \subset \Bo$ be a face and $F \subset B$ its dual. Then for all $p, y \in X$
      \[
	  	\lvert p - y \rvert_E - \lvert p \rvert_E = \lvert p^F - y \rvert_E - \lvert p^F \rvert_E,
      \]
      where as usual $p^F$ denotes the projection of $p$ to $V(F)^\bot = V(E^\circ)^\bot$. 
\end{lem}

\begin{proof}
      Let $\{e_1, \ldots, e_k\}$ be the vertices of $E$ and $\{f_1, \ldots, f_l\}$ those of $F$. Then, by Lemma \ref{lem:samepairing}, there is a $t \in X^*$ such that $\pair{e_i }{ q } = \pair{ t }{ q }$ for all $q \in V(F)$ and all $i \in \{1, \ldots, k \}$. Therefore  
      \begin{align*}
		  \lvert p - y \rvert_E - \lvert p \rvert_E &= -\inf_{i = 1, \ldots, k} \pairl{e_i}{p - y} + \inf_{j = 1, \ldots, k}\pairl{ e_j }{ p } \\
		  &= -\inf_{i} \left[\pairr{ e_i }{ p^F  - y } + \pairr{ e_i }{ p_F } \right] + \inf_{j} \left[\pairr{ e_j }{ p^F } + \pairr{ e_j }{ p_F } \right] \\
		  &= - \inf_{i} \pairr{ e_i }{ p^F  - y }  + \inf_{j} \pairr{ e_j }{ p^F } \\
		  &= \lvert p^F - y \rvert_E - \lvert p^F \rvert_E, 
      \end{align*}
      where the infimum is always taken over $i,j  \in \{1, \ldots, k\}$. 
\end{proof}

\begin{lem} \label{lem:hE1_neq_hE2}
      Let $E_1 \neq E_2$ be two faces of $\Bo$. Then there are no points $p_1, p_2 \in X$ such that $h_{E_1, p_1} = h_{E_2, p_2}$. 
\end{lem}

\begin{proof}
	Without loss of generality, let $\dim E_1 \geq \dim E_2$. Let $u \coloneqq u_1 \in E_1 \setminus E_2$ be a vertex of $E_1$ and $F = \{u\}^\circ \subset B$ its dual facet. By $u_j$, $j = 2, \ldots, r$ we denote the other vertices of $\Bo$. Now we assume there are $p_1, p_2 \in X$ such that $h_{E_1, p_1} = h_{E_2, p_2}$. Then, as $h_{E_1, p_1}(p_1) = \lvert p_1 - p_1 \rvert_{E_1} - \lvert p_1 \rvert_{E_1} = - \lvert p_1 \rvert_{E_1}$, we have
	\begin{align} \label{eq:proof_h_2_neq_h_1}
		- \lvert p_1 \rvert_{E_1} = h_{E_1, p_1}(p_1) =  h_{E_2, p_2}(p_1) = \lvert p_2 - p_1 \rvert_{E_2} - \lvert p_2 \rvert_{E_2}. 
	\end{align}
	Now take $y \in X$ such that $p_1-y, \ p_2-y \in \inte(K_F)$. Since $F$ is a facet of $B$, we have $\dimm(X) = \dimm(K_F)$ and, consequently, we can always find some $y$ big enough satisfying this condition. Then $\fracnorm{p_1 - y} \in \ri(F)$ and 
	\begin{align*}
	  \left \langle u_j \left | \frac{p_1 - y}{\lVert p_1 - y \rVert} \right. \right \rangle \quad \text{ is } \quad
	        \begin{cases}
		    = -1 & \text{if } j = 1, \text{ i.e. } u_j = u \\
		    > -1 & \text{if } j \geq 2.
	        \end{cases}
	\end{align*}
	Thus, as $u$ is a vertex of $E_1$ but not of $E_2$, we obtain by Lemma \ref{lem:||E||=Ei} (with \(E_1^\circ\) as \(F\)):
	\begin{align*}
		h_{E_1, p_1}(y) = \lvert p_1 - y \rvert_{E_1}  - \lvert p_1 \rvert_{E_1} = \lVert p_1 - y \rVert - \lvert p_1 \rvert_{E_1}
	\end{align*}
	and 
	\begin{align*}
		h_{E_2, p_2}(y) &= -\inf_{u_i \in E_2} \langle u_i | p_2 - y \rangle - \lvert p_2 \rvert_{E_2} \\
		&= -\inf_{u_i \in E_2} \big [ \langle u_i | p_2 - p_1 \rangle + \langle u_i | p_1 - y \rangle \big] - \lvert p_2 \rvert_{E_2} \\
		&< \lVert p_1 - y \rVert + \lvert p_2 - p_1 \rvert_{E_2} - \lvert p_2 \rvert_{E_2} \\
		&\overset{(\ref{eq:proof_h_2_neq_h_1})}{=}\lVert p_1 - y \rVert - \lvert p_1 \rvert_{E_1} = h_{E_1, p_1}(y). 
	\end{align*}
	Hence for every pair $p_1, p_2 \in X$ we have found a point where $h_{E_1, p_1}$ and $h_{E_2, p_2}$ do not coincide, which is a contradiction to our assumption. 
\end{proof}

\section{The horofunction compactification} \label{sec:horofunction}

\subsection{Introduction to horofunctions}

For this general introduction let $(X,d)$ be a quasi-metric space, that is, \((X,d)\) is a metric space without the condition of symmetry, thus $d(x,y) \neq d(y,x)$ for any $x, y \in X$ is possible. We assume the topology on \(X\) to be defined via the symmetrized distance
\[
      d_{sym}(x,y) \coloneqq d(x,y) + d(y,x)
\]
for all $x,y \in X$. 
Let $p_0$ be a basepoint and let $C(X)$, the space of continuous real-valued functions on $X$, be endowed with the topology of uniform convergence on bounded subsets. 
We denote its quotient by constant functions by $\widetilde{C}(X)$. The horofunction compactification of $X$ is a certain embedding of $X$ as an open and dense subset of a certain compact space in  $\widetilde{C}(X)$. To obtain this embedding we define

\begin{equation*}
	\begin{aligned}
		\psi: X &\lora \widetilde{C}(X) \nonumber \\ 
		z &\longmapsto \psi_z = d(\cdot, z) - d(p_0,z).
	\end{aligned}
\end{equation*}
By using the triangle inequality it can be shown that this map is injective and continuous. If X is geodesic, proper with respect to $d_{sym}$ and if $d$ is symmetric with respect to convergence\footnote{Symmetry of $d$ with respect to convergence means that for any sequence $(x_n)_n \subset X$ and $x \in X$ it holds $d(x_n, x) \ra 0$ if and only if $d(x, x_n) \ra 0$.}, then $\psi$ is an embedding and the image $\psi(X) \subset \widetilde{C}(X)$ is relatively compact. For more details of this construction see \cite[p.4 and Prop. 2.2.]{wa1}.

\begin{defi}
      The horofunction boundary of a metric space $X$ is the boundary of the closure of the image of the map $\psi$ in $\widetilde{C}(X)$:
      \[
	  	\delhor (X) \coloneqq \Big(\cl \psi(X) \Big) \setminus \psi(X).
      \]
      Its elements are called \emph{horofunctions}. If the closure $\overline{X}^{hor} \coloneqq X \cup \delhor X$ is compact, it is called the \em{horofunction compactification} of $X$. 
\end{defi}

\begin{rem}\cite{wa1} \
	\begin{enumerate}
		\item The choice of an alternative basepoint $p_0'$ leads to a homeomorphic boundary and compactification.  
		\item All elements of $\cl \psi(X)$ are 1-Lipschitz with respect to $d_{sym}$. 
	\end{enumerate}
\end{rem}

From now on we assume all conditions to be satisfied such that $\psi$ is an embedding with relatively compact image. Identify $X$ with $\psi(X)$. We say that a sequence $(z_n)_{n \in \NN} \subset X$ converges to a horofunction $\xi \in \delhor X$ if the sequence of associated functions compactly converges, that is, it converges uniformly on compact subsets: $\psi_{z_n} \lora \xi$. Rieffel \cite[Lemma 4.5]{ri} showed that there are special sequences that always converge to a horofunction $\xi \in \delhor X$, namely those along so-called almost-geodesics:

\begin{defi}
      A continuous map $\gamma: \RR \lora X$ is called an {\em almost-geodesic} if for every $\varepsilon > 0$ there is an $N \in \NN$ such that for all $ t \geq s \geq N$ 
      \[
	  \lvert d\big(\gamma(0), \gamma(s)\big) + d\big(\gamma(s), \gamma(t)\big) - t \rvert < \varepsilon.
      \]
\end{defi}

\begin{lem}\cite[Lemma 4.5]{ri}
	For every almost geodesic \(\gamma: \RR \lora \NN\) and every sequence \((t_n)_{n \in \NN}\) diverging to \(\infty\), there is a limit 
	\[
		\xi = \lim_{n \ra \infty}\psi_{\gamma(t_n)} \in \delhor X.
	\]
	The limit is independent of the choice of the sequence \((t_n)_n\) and we call it the \emph{limit} of the almost-geodesic \(\gamma\). 
\end{lem}

\begin{defi}
      A horofunction which is the limit of an almost-geodesic is called a \emph{Busemann point}. 
\end{defi}

In general, not all horofunctions have to be Busemann points, and it is an interesting question when this actually happens. In the case of a finite-dimensional vector space with polyhedral norm we know by Walsh \cite[Thm. 1.2]{wa2} that actually all horofunctions are Busemann points. To calculate them, Walsh first defines certain functions \(f_{E,p}\), which depend on proper faces $E \subset \Bo$ and points $p \in X$. They are defined by 
\begin{align} \label{legendre}
	f_{E,p}: \ X^* &\lora [0, \infty], \nonumber \\
	q &\longmapsto f_{E,p}(q) \coloneqq I_E(q) + \langle q | p \rangle - \inf_{y \in E} \langle y | p \rangle,
\end{align}
where the indicator function $I_E(q)$ is $0$ for $q \in E$ and $\infty$ elsewhere. 
The Busemann points are then given as the Legendre-Fenchel-transforms $f_{E,p}^*$ of the function \(f_{E,p}\). Recall that the Legendre-Fenchel-transform $g^*$ of a function $g:X \ra \RR \cup \{\infty\}$ is given by 
\begin{align*}
      g^*: X^* &\lora \RR \cup \{\infty\},\\
      w &\longmapsto \sup_{x \in X} \big(\langle w | x \rangle - g(x) \big).
\end{align*}
More details about Legendre-Fenchel-transforms can be found for example in \cite[\S 7.2]{be}.

The result of Walsh can be stated as follows:

\begin{thm}[\cite{wa2}, Thm. 1.1.]
      Let $(X, \normm)$ be a finite-dimensional vector space with polyhedral norm and let the notation be as above. Then the set of Busemann points is the set 
      \[
	  \{f_{E,p}^* \ |  E \subset \Bo \text{ is a (proper) face, } p \in X\}.
      \]
\end{thm}

We now show that our previously defined maps $h_{E,p}$ agree with the maps \(f_{E,p}^*\). 

\begin{lem} \label{lem:fep-pseudonorm}
      Let $E$ be a face of $\Bo$ and $p \in V$. Then 
      \[
        f_{E,p}^*(\cdot) = h_{E,p}(\cdot) = |p - \cdot|_E - |p|_E.
      \] 
\end{lem}

\begin{proof}
      By definition we obtain for all $y \in V$:
      \begin{align*}
        f_{E,p}^*(y) &= \sup_{x \in X^*}[\langle x|y\rangle - f_{E,p}(x)]\\
        &= \sup_{x \in X^*}[\langle x|y\rangle - I_E(x) - \langle x|p\rangle + \inf_{q \in E}\langle q|p\rangle]\\
        &= \sup_{x \in E}[\langle x|y-p\rangle] + \inf_{q \in E}\langle q|p\rangle\\
        &= -\inf_{x \in E}(\langle x|p-y\rangle) + \inf_{q \in E}\langle q|p\rangle\\
        &= |p-y|_E - |p|_E. \qedhere
      \end{align*}
\end{proof}

\begin{cor}
      With the notations as in the previous lemma, the following holds: 
      \[
	  f_{E,p}^* = f_{E,p^F}^*.
      \]
\end{cor}

\begin{proof}
      The statement follows directly by Lemma \ref{lem:p^Fonly}.
\end{proof}

In summary we can describe the set of horofunctions as
      \[
	  \del_{hor}X = \left \{h_{E,p} \ |  E \subset \Bo \text{ is a proper face, } p \in V(E^\circ)^\bot \right\}.
      \]

In the following section we describe the topology of $\overline{X}^{hor}$.

\subsection{The characterization theorem}


The main theorem of this section characterizes all sequences converging to a horofunction. It shows the strong relation between the topology of the set of horofunctions and the shape of the dual unit ball, which is the underlying principle of the homeomorphism in Theorem \ref{thm:homeo}. This result is also used in \cite{js} to establish a geometric one-to-one correspondence between the non-negative part of $n$-dimensional projective toric varieties and horofunction compactifications of $\RR^n$ with respect to rational polyhedral norms.

Before we state the theorem to characterize converging sequences, we prove a lemma which already contains the main idea of the characterization (Lemma \ref{lem:subsequence}) and a lemma that simplifies the proof of the main theorem (Lemma \ref{lem:shift_in_cone}). 

\begin{lem} \label{lem:subsequence}
      Let $(x_n)_{n \in \NN}$ be an diverging sequence in $(X, \normm)$ and let $B$ be the polyhedral unit ball associated to $\normm$. Then there is a proper face $F \subset B$ and a point $p \in V(F)^\bot$ such that the sequence $(x_n)_n$ has a subsequence $(x_{n_j})_j$ which satisfies the following properties:
      \begin{itemize} 
        \item[(i)]  the projections $x_{n_j, F}$ lie in $K_{F}$ for all $j \in \NN$.
        \item[(ii)] the distance of the projection \(x_{n_j, F}\) to the relative boundary diverges: $d \left(x_{n_j,F} ,\del_{rel} K_F \right ) \lora \infty$ as $n_j \lora \infty$.
        \item[(iii)] the sequence of orthogonal projection converges: $\lVert x_{n_j}^F - p \rVert \lora 0$ as $n_j \lora \infty$. 
      \end{itemize}
\end{lem}

\begin{proof}
      To find a face \(F\) of $B$ and a subsequence of \((x_n)_n\) satisfying all three properties, we start to look at the facets, the top-dimensional faces. As $B$ is a polyhedral unit ball, it has only finitely many of them and their cones cover the whole vector space $X$. Therefore we find a subsequence of \((x_n)_n\), also denoted by $(x_n)_{n}$, and a facet $F \subset B$ such that $x_n \in K_F$ for all $n$. As \(F\) is a facet, we know that $V(F) = X$ and therefore the projection is the identity and properties \((i)\) and \((iii)\) are satisfied with \(p = 0\). If property \((ii)\) also holds, then we are done. 
      
      So let us assume the distance \(d(x_n, \delrel K_F)\) is bounded for our facet \(F\). Then there is at least one face in the relative boundary $\delrel K_F$ and a subsequence of \((x_n)_n\) at bounded distance to the cone over this one face. We replace the sequence \((x_n)_n\) by this subsequence and denote it by \((x_m)_{m \in M}\) with \(M \subset \NN\). Take the intersection of all faces in $\delrel K_F$ to which this particular subsequence \((x_m)_m\) has bounded distance.  As the  boundary of the cone \(K_F\) is the union of cones all intersecting in the origin, these cones have unbounded increasing distance to each other if they do not have a common subspace. Because only those cones contribute to the intersection that are in bounded distance to \((x_m)_m\), and thereby also to each other, this intersection is not trivial and it is again a cone generated by a face $F_1$ of $B$ with $\dim(F_1) < \dim(F)$. 
      
      Consider now the projections \((x_{m,{F_1}})_m\) and \((x_m^{F_1})_m\) of the sequence \((x_m)_m\) to the subspace $V(F_1)$ and \(V(F_1)^\bot\), respectively. As the distance of \((x_m)_m\) to $V(F_1)$ is bounded by construction, the orthogonally projected sequence $(x_m^{F_1})_m \subset V(F_1)^\bot$ lies within a compact subset and so there is a subsequence $(x_{m_j})_{j \in \NN} \subset (x_m)_{m \in M}$ such that the orthogonal part \(x_{m_j}^{F_1}\) converges to a point  $p_1 \in V(F_1)^\bot$. Thus the sequence \((x_{m_j})_j\) satisfies the third property of the lemma with respect to $F_1$. By the minimality condition coming from the construction of \(F_1\) above, we guarantee that the distance of the projected sequence $(x_{m_j, F_1})_j$ to $\delrel K_{F_1}$ is unbounded, which implies that the  sequence \((x_{m_j})_j\) satisfies property \((ii)\) of the lemma.
      
      \begin{figure}[h!]
      	\includegraphics[scale=0.89]{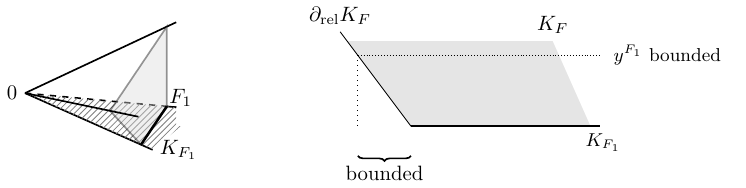}
      	\caption{\textsc{left}: A schematic picture of the cone \(K_F\) with \(K_{F_1}\) in its boundary. \textsc{right}: The projections of points with \(y^{F_1}\) bounded remain within bounded distance to \(\delrel K_{F_1}\) (view from the origin into \(K_F\)).} \label{pic:proof_characterizationthm}
      \end{figure}

      It remains to show that we find a subsequence of $(x_{m_j, F_1})_j$ lying within the cone \(K_{F_1}\). See Figure \ref{pic:proof_characterizationthm} for an idea of the following proof: By the third property of the lemma, the sequence \((x_{m_j}^{F_1})_j\) is bounded, that is, there is some \(b \in \RR\) such that \(\norm{x_{m_j}^{F_1}} \leq b\) for all \(j \in \NN\).  Consider the set \(S \coloneqq \left\{y \in  K_F | \lVert y^{F_1} \rVert \leq b \right\}\). Its orthogonal projection to \(V(F_1)\) remains within bounded distance to \(K_{F_1}\). As \((x_{m_j, F_1})_j\) diverges from \(\delrel K_{F_1}\) by property \((ii)\), there must be infinitely many points of the sequence \((x_{m_j,F_1})_j\) lying within \(K_{F_1}\). We choose the corresponding points of \((x_{m_j})_j\) as our new subsequence, and have thereby found a subsequence of \((x_n)_n\) satisfying all three required properties. 
\end{proof}

\begin{lem} \label{lem:shift_in_cone}
	Let $B \subset (X, \normm)$ be a polyhedral unit ball associated to $\normm$. Let \((x_n)_{n \in \NN} \subset X\) be an diverging sequence such that there is a face \(F \subset B\) with respect to which the following properties hold: 
	\begin{itemize}
		\item[(i)] \(x_n \in K_F\) for \(n\) large enough.
		\item[(ii)] \(d(x_n, \delrel K_F) \lora \infty\) as \(n \lora \infty\). 
	\end{itemize}
	Then for every \(y \in X\) there is an \(N \in \NN\) such that for every \(n \geq N\) the point \(x_n + y\) is contained in the closed cone over a facet which contains \(F\) in its relative boundary. If \(F\) is itself a facet, then \(x_n + y \in K_F\) for \(n\) sufficiently large. 
\end{lem}

\begin{proof}
	As \(x_n \in K_F\) for large \(n\) and the distance of the sequence \((x_n)_n\) to the relative boundary \(\delrel K_F\) goes to infinity, we also have \(x_n + y_F \in K_F\) for possibly some larger \(n\). It remains to show that when adding \(y^F \in V(F)^\bot\) with \(y^F \neq 0\), we obtain a point in the cone over a facet that contains \(F\) in its relative boundary. 
	
	Let \(F_j\) denote some facet of \(B\) that has \(F\) in its relative boundary and let \(S \coloneqq \left\{z \in K_{F_j} \ | \ d(z, K_F) = b\right\}\) be the set of all points in the cone \(K_{F_j}\) that have distance \(b>0\) to the cone \(K_F\) for some fixed \(b \in \RR_+\). Then the orthogonal projection \(S_F\) of \(S\) to the subspace \(V(F)\) is a possibly truncated parallel translate of the cone \(K_F\).  In particular, the relative boundary of the projection \(S_F\) stays within bounded distance to the relative boundary of \(K_F\) as a consequence of triangle inequalities. In other words: take a point of the form \(x + h\), where \(x \in K_F\) and \(h \in V(F)^\bot\) with \(\norm{h} = b \neq 0\),  such that the point \(x + h\) is not contained in the relative interior of \(F_j\). Then this point lies within a \(\beta_j\)-tubular neighborhood of \(\delrel K_F\), where \(\beta_j \in \RR\) only depends on \(F_j\) and \(b\).  
	
	Back to our sequence \((x_n)_n\). As the distance \(d(x_n + y_F, \delrel K_F)\) goes to infinity by property \((ii)\), we know by the discussion in the previous paragraph that the point \(x_n + y = x_n + y_F + y^F\) with \(y^F \neq 0\) has to be inside a cone over a facet which has \(F\) in its relative boundary for \(n\) sufficiently large (note that \(y^F \in V(F)^\bot\) is constant). 
\end{proof}
 

We are now prepared to state and prove the main theorem of this section. 

\begin{thm} \label{thm:characterization}
      Let $B \subset (X, \normm)$ be a convex polyhedral unit ball associated to $\normm$ and $\Bo \subset X^*$ its dual. For any sequence $(z_n)_{n \in \NN} \subset X$ the associated sequence \((\psi_{z_n})_{n \in \NN}\) with $\psi_{z_n} (\cdot) = \lVert z_n - \cdot \rVert - \lVert z_n \rVert$ converges to a horofunction $h_{E,p}$ with $E$ a proper face of $\Bo$ and $p \in V(E^\circ)^\bot$, if and only if the following properties are satisfied with respect to the proper face $F = E^\circ \subset B$: 
      \begin{itemize}
	  \item[(0)]  The sequence diverges: $\lVert z_n \rVert \lora \infty$ as $n \lora \infty$.
	  \item[(i)]  The projection of the sequence \((z_n)_n\) to $V(F)$ lies in the cone over $F$: $ z_{n,F} \in K_{F}$ for $n$ big enough.
	  \item[(ii)] The distance of the projection of the sequence to the relative boundary of the cone diverges: $d \big(z_{n,F}, \delrel K_F \big) \lora \infty$ as $n \lora \infty$.
	  \item[(iii)] The orthogonal projection of the sequence is bounded and converges to $p \in V(F)^\bot$: 
	$\lVert z_n^F - p \rVert \lora 0$ as $n \lora \infty$.  
      \end{itemize}
\end{thm}

\begin{proof}
      We first show that $(\psi_{z_n})_{n \in \NN}$ converges to $h_{E,p}$ if all properties are satisfied. Let $(z_n)_{n \in \NN}$ be a sequence satisfying all properties for some face $F$ of $B$ and $p \in V(F)^\bot$. By Lemma \ref{lem:shift_in_cone} we know that for every $y \in X$ and every \(n \in \NN\) sufficiently large, each of the two points \(z_{n,F} + p - y\) and \(z_{n,F} + p\) either lies in the cone over a facet that has \(F\) in its relative boundary or in the cone over \(F\). Therefore we can apply Lemma \ref{lem:||E||=Ei} and Lemma \ref{lem:||E=||B+||E} to compute
      \begin{align*}
	  (\psi_{z_n} &- h_{E,p})(y) = \lVert z_n - y \rVert - \lVert z_n \rVert - h_{E,p}(y)\\
	  &= \lVert z_n^F  - p + z_{n,F} + p - y \rVert - \lVert z_n^F  - p + z_{n,F} + p \rVert - h_{E,p}(y) \\
	  &\leq \lVert z_n^F  - p \rVert + \lVert z_{n,F} + p - y \rVert  + \lVert z_n^F  - p \rVert  - \lVert z_{n,F} + p \rVert - h_{E,p}(y) \\
	  &\overset{\ref{lem:||E||=Ei}}{=} |z_{n,F} + p - y|_E - |z_{n,F} + p|_E - h_{E,p}(y)  + 2 \lVert z_n^F  - p \rVert \\
	  &\overset{\ref{lem:||E=||B+||E}}{=} \lVert z_{n,F} \rVert + |p - y|_E  - \lVert z_{n,F} \rVert - |p|_E - h_{E,p}(y)  + 2 \lVert z_n^F  - p \rVert \\
	  &= |p-y|_E -  |p|_E - h_{E,p}(y)  + 2 \lVert z_n^F  - p \rVert \\	  
	  &\lora 0
      \end{align*}
      by the usual and the reverse triangle inequality. Similarly we get 
      \begin{align*}
	  (\psi_{z_n} &- h_{E,p})(y) = \lVert z_n^F - p + z_{n,F} + p - y \rVert - \lVert z_n^F  - p + z_{n,F} + p \rVert - h_{E,p}(y) \\
	  &\geq -\lVert z_n^F  - p \rVert + \lVert z_{n,F} + p - y \rVert  - \lVert z_n^F - p \rVert - \lVert z_{n,F} + p \rVert - h_{E,p}(y) \\
	   &\lora 0.
      \end{align*}
       Thus we have shown that $\psi_{z_n}(y) \lora h_{E,p}(y)$ for all $y \in X$. As we assume $d_{sym}$ to be proper and because all elements of $\cl\{\psi_z \ | z \in X\}$ are 1-Lipschitz with respect to $d_{sym}$, the pointwise convergence of $\psi_{z_n}$ is equivalent to the uniform convergence on bounded sets, which again is equivalent to the uniform convergence on compact sets in $C(X)$. See for example \cite{wa2} for a reference. 
      Therefore $\psi_{z_n} \lora h_{E,p}$.
           
      For the other direction we have to show that every converging sequence $(z_n)_{n \in \NN} \subset X$ with $\psi_{z_n} \lora h_{E,p}$ for some proper face $E \subset \Bo$ and $p \in V(E^\circ)^\bot$ satisfies the properties of the theorem. The proof is based on Lemma \ref{lem:subsequence}, where we have shown that every diverging sequence has a subsequence fulfilling  properties $(i) - (iii)$ for some proper face $F \subset B$. 
      
      So let $(z_n)_{n \in \NN}$ be a sequence with $\psi_{z_n} \lora h_{E,p}$ and let $F\coloneqq E^\circ$ be the dual face. If the sequence \((z_n)_{n \in \NN}\) were bounded, $\psi_{z_n}$ would stay in the interior of $\psi(X)$ and not converge to a Busemann point in the boundary. Thus, \((z_n)_{n \in \NN}\) diverges and, by Lemma \ref{lem:subsequence}, $(z_n)_{n\in \NN}$ has a subsequence $(z_{n_j})_{j\in \NN}$ satisfying all properties with respect to $F$. We have to show that this subsequence is the entire sequence except for finitely many points. If it was not almost the entire sequence, we could find another (diverging) subsequence $(z_{n_k})_{k \in \NN}$ of \((z_n)_{n \in \NN}\) satisfying all properties of Lemma \ref{lem:subsequence} for some face $F_1 \neq F$ and some point \(p_1 \in V(F_1)^\bot\). Then by the first part of the proof we would have $\psi_{z_{n_k}} \lora h_{E_1,p_1} \neq h_{E,p_1}$ as $E_1 \neq E$ (see Lemma \ref{lem:hE1_neq_hE2}) which is a contradiction. The same argument works if we had a subsequence fulfilling the properties from Lemma \ref{lem:subsequence} for some $p_1 \neq p$ with $p_1 - p \notin V(F)$. If \((z_{n})_{n \in \NN}\) had a subsequence \((z_{n_k})_{k \in \NN}\) which was bounded, \((\psi_{n_k})_{k \in \NN} \subset (\psi_{z_n})_{n \in \NN}\) would not converge to \(h_{E,p}\) as required. 
\end{proof}


\subsection{Examples} \label{sec:examples}
Let us look at some examples to illustrate the assumptions of the theorem and to give the reader some intuition how sequences converge. In all examples below we consider $\RR^2$ equipped with the $L^1$-norm. Its dual is the $L^\infty$-norm as seen in Example \ref{ex:dualL1} before. The unit ball $B$ and its dual $\Bo$ as well as the notation of faces are shown in Figure \ref{fig:B_Bo_examples}. 

\begin{figure}[h!]
      \includegraphics{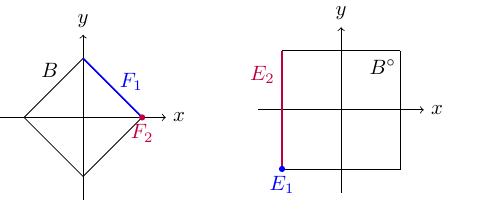}
      \caption{The unit ball $B$ (\textsc{left}) and its dual $\Bo$ (\textsc{right}) with some faces. The face $F_1$ corresponds to the point $E_1$ while $F_2$ is dual to the facet $E_2$.}\label{fig:B_Bo_examples}
\end{figure}

We consider sequences of the form $z_n = (n, f(n)) \in \RR^2$ following functions $f: \RR \ra \RR$. These functions are shown in Figure \ref{fig:function_examples}.

\begin{ex} \label{ex:1}
      For some constant $c \in \RR$ and $n \in \NN$ consider the sequence 
      \[
	  	z_n = (n, c) \in \RR^2.
      \]
       The sequence runs along a line of ``height'' $c$ parallel to the \(x\)-axis. The cone $K_{F_2}$ over the face \(F_2 \subset B\) is the non-negative \(x\)-axis with the origin as its relative boundary.  $V(F_2)^\bot$ then is the \(y\)-axis isomorphic to $\RR$. It is easy to see that all properties of the theorem are satisfied with $F = F_2$ and $p = c$. 
		Note that, as the sequence is parallel to the \(x\)-axis, it is not possible to choose $F_1$ as face here because the distance to the relative boundary \(\delrel K_{F_1}\) is constant. As $F_2^\circ = E_2$ we conclude that $\psi_{z_n} \lora h_{E_2, c}$.
\end{ex}

\begin{ex} \label{ex:2}
      Next we consider sequences $(z_n)_{n \in \NN}$ of the type 
      \[
	  z_n = (n, sn) \in \RR^2
      \]
      with $s > 0$.  Now we choose $F_1$ as our face. As the cone \(K_{F_1}\) is the first quadrant with the non-negative \(x\)- and \(y-\)axis as its relative boundary, the distance \(d(z_n, \delrel K_{F_1})\) diverges because \(s \neq 0\). 
	The dual $E_1$ of $F_1$ is just a point, $E_1 = \{e_1\}$, and, by Equation (\ref{legendre}) on page \pageref{legendre}, it is clear that $h_{E_1,p}(x) = \langle e_1  | x \rangle$ is independent of $p$ for all $x \in \RR^2$. 
	This fits with our conditions above, since $V(F_1) = \RR^2$ and thus $V(F_1)^\bot = \{0\}$. 
	The convergence of $(z_n)_n$ is independent of the value of $s$, all sequences of this type converge to the same horofunction $h_{E_1}$. 
\end{ex}

\begin{ex} \label{ex:3}
      We now take a sequence that lies completely in $K_{F_1}$ but converges to the horofunction associated to the face $F_2$. For $n \in \NN$ let
      \[
	  z_n = \left (n, \frac{1}{n} \right) \in \RR^2 
      \]
      be our sequence. As $(z_n)_n$ approaches the $x$-axis for \(n \lora \infty\) the boundary condition $(ii)$ is not satisfied for $F = F_1$. If we take $F = F_2$ instead, this is not a problem anymore. As not the whole sequence but only its projection \((z_{n,F})_n\) to the subspace \(V(F)\) has to be inside the cone \(K_F\), it is easy to check that all requirements are fulfilled for $F_2$  with $p = 0$. Hence we get the convergence \(\psi_{z_n} \lora h_{E_2,0}\), which is the same limit as for the sequence in the first example with $c = 0$. 
\end{ex}

\begin{figure}[h!]
      \includegraphics{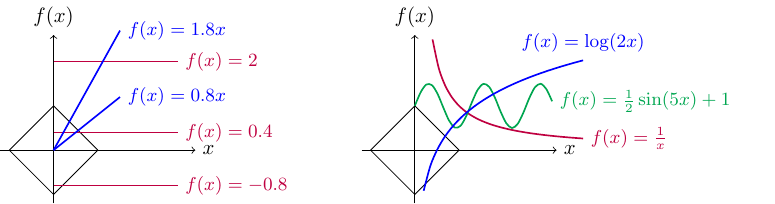}
      \caption{The functions for Examples \ref{ex:1} and \ref{ex:2} (\textsc{left}) and \ref{ex:3} to \ref{ex:5} (\textsc{right})}\label{fig:function_examples}
\end{figure}

\begin{ex} \label{ex:4}
      The next sequence $(z_n)_{n \in \NN}$ we consider is given by
      \[
	  z_n = \left(n, \frac{1}{2} \sin(5n) + 1 \right) \in \RR^2. 
      \]
      This sequence lies completely in $K_{F_1}$, but as the $y$-value is bounded, property \((ii)\) about the distance to the relative boundary is not satisfied with respect to \(F_1\). If we choose $F_2$ instead, this property is satisfied but we are not able to find an appropriate constant $p \in V(F_2)^\bot \simeq \RR$ to fulfill the convergence of the orthogonal sequence \(\left(z^{F_2}\right)_n\) (i.e. property \((iii)\)). As these two faces are the only reasonable choices, we conclude that $\psi_{z_n}$ does not converge in this case. This also turns out when doing the calculation explicitly. 
\end{ex}

\begin{rem}
	One could guess that an easier condition for finding the appropriate face $F$ is to look at the limiting direction $\left(\frac{z_n}{\lVert z_n \rVert}\right)_n$ and to require it to be in $F$. The next example shows that this expectation is too simplistic. 
\end{rem}

\begin{ex} \label{ex:5}
       For $n \in \NN$ take the sequence \((z_n)_{n \in \NN}\) with  
      \[
	  		z_n = \big(n, \log(2n)\big) \in \RR^2.
      \]
      Then we have $\frac{z_n}{\lVert z_n \rVert} \lora (1,0) \in F_2$, thus following the previous remark, we first try $F = F_2$. By considering property \((ii)\) we arrive at a problem, because the sequence $\left(z_n^{F_2}\right)_n = (0, \log(2n))_n \subset \{0\} \times \RR$ does not converge. If we take $F = F_1$ instead, the divergence of the second component guaranties us unbounded distance to the relative boundary (i.e. property \((ii)\)). As \(F_1\) is a facet, the projection is the identity and \(V(F_1)^\bot = \{0\}\). Thus all requirements are fulfilled with respect to \(F_1\) and we conclude $\psi_{z_n} \lora h_{E_1}$, independent of $p$ as explained before in Example \ref{ex:2}.  
\end{ex}

\begin{rem}
      We have seen in the examples above that it is not enough to consider the direction of a sequence to determine the right face \(F\) of \(B\). The easiest examples to consider are sequences following straight lines and they are also suited to show the general converging behavior: all sequences in a regular direction\footnote{A sequence goes in a regular direction if it follows a line parallel to a line through the origin and within the interior of the cone over a facet.} collapse and converge to the horofunction associated to the dual vertex, independent of any translation or direction. For a sequence in  a singular direction associated to a lower dimensional face $F$ we have the same collapsing behavior for the $z_{n,F}$-part and a blowing-up in the orthogonal direction $V(F)^\bot$, where each point $p \in \Fbot$ leads to a different boundary point $h_{E,p}$.  
\end{rem}


\section{Proof of Theorem \ref{thm:homeo}} \label{sec:proof_of_thm}

In this last part of this paper we prove Theorem \ref{thm:homeo}, which yields an explicit homeomorphism between the dual unit ball \(\Bo\) and the compactification \(\xhor\) that respects the stratification of the boundary of \(\xhor\) through the facial structure of \(\Bo\). This theorem thereby answers the question posed by by Kapovich and Leeb in \cite[Quest. 6.18]{kl}.

To construct the homeomorphism we need a map $m^C$ that maps a finite-dimensional vector space to the interior of a convex polytope $C$ of the same dimension. The structure of the map \(m^C\) is motivated by the moment map known from the theory of toric varieties. See for example \cite[\S 4.2]{fu} for a description of it. Although we do not have a Lie group acting on a toric variety here, the map is the same. Up to some signs which come from the definition of the dual unit ball, the same result as Theorem \ref{thm:mc} can be found in \cite[p. 82]{fu} but with a different proof. The moment map was also used to realize the closure of a flat in the Satake compactifications as bounded polytopes in \cite{ji}. 

\subsection{The map $m^C$} \label{sec:m^c}

 We define the map $m^C$ on the real vector space $\RR^k$ and its dual $\left(\RR^{k}\right)^*$ in order to be able to use methods from analysis for the proof\footnote{By choosing a basis and identifying \(X\) with \(\RR^m\), Theorem \ref{thm:mc} also holds for $m^C$ defined on $X$ and $C \subset X^*$.}. Let $C \subset \left(\RR^{k}\right)^*$ be a $k$-dimensional polytope with vertices $\{c_1, \ldots, c_r\}$. Set
\begin{equation*} \label{eq:mc}
      \begin{aligned} 
	m^C: \RR^k &\lora \inte(C),\\
	x &\longmapsto \sum_{i = 1}^r \frac{e^{-\langle c_i  | x \rangle}}{\sum_{k =1 }^r e^{-\langle c_k  | x \rangle}} c_i. 
      \end{aligned}
\end{equation*}

Let $x \in \RR^k$ be an arbitrary point. As no summand $e^{-\langle c_i | x \rangle}$ in the numerator can be zero, $m^C(x)$ is a proper convex combination of all vertices of $C$ and lies therefore in its interior. It is obvious that the map is continuous and twice differentiable. Bijectivity of the map will be shown later in Theorem \ref{thm:mc}.

As already done in the proof of Lemma \ref{lem:pairinginfty} we want to use $C^\circ$ as the unit ball of a subspace. Since $C$ does not necessarily contain the origin in its interior, $(C^\circ)^\circ$ is not $C$ anymore, which was a strong assumption in the duality between faces of \(B\) and of \(\Bo\). Thus we can not take $C^\circ$ as the unit ball of a subspace, but we have to consider a translated version \(C_s\) of \(C\) whose interior contains the origin, and then take the dual of this set. The following lemma shows that the map $m^C$ behaves well under such a shifting.

\begin{lem} \label{lem:mcshift}
	Let $C_s = C + s \subset (\RR^k)^*$ be the convex polyhedral set obtained by shifting $C$ by an element $s \in \left(\RR^k\right)^*$. Then for all $x \in \RR^k$
	\begin{align*}
		m^{C_s}(x) = m^C(x) + s.
	\end{align*}
\end{lem}

\begin{proof}
	Let $x \in \RR^k$ be arbitrary. Then, as $C_s = \conv\{c_j + s | j = 1, \ldots, r \}$, we have
	\begin{align*}
		m^{C_s}(x) &= \sum_{i = 1}^r \frac{e^{-\langle c_i + s | x \rangle}}{\sum_{k =1 }^r e^{-\langle c_k + s | x \rangle}} (c_i + s) \\
		&= \sum_{i = 1}^r \frac{e^{-\langle c_i | x \rangle - \langle s | x \rangle}}{\sum_{k =1 }^r e^{-\langle c_k | x \rangle - \langle s | x \rangle}} (c_i + s) \\
		&= \sum_{i = 1}^r \frac{e^{-\langle c_i | x \rangle}}{\sum_{k =1 }^r e^{-\langle c_k | x \rangle}} c_i + \sum_{i = 1}^r \frac{e^{-\langle c_i | x \rangle}}{\sum_{k =1 }^r e^{-\langle c_k | x \rangle}} s\\\
		&= m^C(x) + s. \qedhere
	\end{align*}
\end{proof}

To prove injectivity of the map \(m^C\) we need some knowledge about convex and concave functions, which we will now briefly repeat. More details can be found for example in \cite[\S 3.1]{BV}. 

\begin{defi}
 A function \(f: D \subset \RR^k \lora \RR\) is called \emph{convex} if \(D\) is a convex set and if 
 \[
    f(\lambda x + (1-\lambda)y) \leq \lambda f(x) + (1-\lambda)f(y),
 \]
    for all \(x,y \in D\) and \(0 \leq \lambda \leq 1\). 
    
    A function \(f\) is called \emph{strictly convex}, if we have \(<\) in the above inequality whenever \(x \neq y\) and \(0 < \lambda < 1\). 
    
    A function \(f\) is called \emph{(strictly) concave}, if \(-f\) is (strictly) convex. 
\end{defi}

\begin{lem} \cite[\S 3.1]{BV} \label{lem:convex_properties}
    Let \(f:D \subset \RR^k \lora \RR\) be a function and \(D \subset \RR^k\) an open convex set. Then we have the following characterizations for concavity:
    \begin{enumerate}
     \item First-order condition: Assume \(f\) is differentiable\footnote{Hereby we mean that the gradient \(\nabla f\) exists at every point in \(D\).}. Then \(f\) is strictly concave if and only if 
     \[
        f(y) < f(x) + \nabla f(x)^T (y-x).
     \]
     \item Second-order condition: Assume that \(f\) is twice differentiable\footnote{Hereby we mean that the Hessian \(\nabla^2 f\) exists at every point in \(D\).}. Then \(f\) is strictly concave if and only if the Hessian of \(f\) is negative definite, that is, for all \(x \in D\) we have
     \[
        \nabla^2 f(x) \leq 0.
     \]
    \end{enumerate}
\end{lem}

\begin{lem} \label{lem:derivative_mC_neg_def}
    The first derivative of \(m^C\) is negative definite at every \(x \in \RR^k\). 
\end{lem}

\begin{proof}
     Recall that 
	  \[
	  	m^C(x) = \frac{\sum_{i = 1}^r e^{-\langle c_i | x \rangle} c_i}{\sum_{k = 1}^r e^{-\langle c_k | x \rangle}} \in \left(\RR^k\right)^*
	  \] 
	  for all $x \in \RR^k$.  We use the notation that an upper index denotes the corresponding component of the vector. Then 
	  
	  \begin{align*}
	  	\frac{\del \left(m^C\right)^\alpha}{\del x^\beta} (x) 
	 	 &= \frac{1}{(\sum_{j = 1}^r e^{-\langle c_j | x \rangle})^2} \left [ \sum_{i,j= 1}^r c_i^\alpha c_j^\beta e^{-\langle c_i + c_j|x \rangle} - \sum_{i,j = 1}^r c_i^\alpha c_i^\beta e^{-\langle c_i + c_j|x \rangle} \right ]\\
	  	&= \frac{1}{(\sum_{j = 1}^r e^{-\langle c_j | x \rangle})^2} \sum_{i < j } e^{-\langle c_i + c_j | x \rangle} \left( c_i^\alpha c_j^\beta + c_j^\alpha c_i^\beta - c_i^\alpha c_i^\beta - c_j^\beta c_j^\alpha\right) \\
	  	&= \frac{-1}{(\sum_{j = 1}^r e^{-\langle c_j | x \rangle})^2} \sum_{i < j } e^{-\langle c_i + c_j | x \rangle} \left( c_i^\alpha - c_j^\alpha\right)\left(c_i^\beta - c_j^\beta\right). 
	  \end{align*}
	  
	  To show that the derivative of the map $m^C$ is negative definite, let $a = (a^1 \ldots a^k)^T \in \RR^k$ be some arbitrary vector. We compute
	  
	  \begin{align*}	  
	  	(a^1 \ldots a^k)& 
	  	\begin{pmatrix}
	  		\frac{\del \left(m^C\right)^1}{\del x^1} (x) & \cdots & \frac{\del \left(m^C\right)^1}{\del x^k} (x) \\
	  		\vdots & \ddots & \vdots \\
	  		\frac{\del \left(m^C\right)^k}{\del x^1} (x) & \cdots & \frac{\del \left(m^C\right)^k}{\del x^k} (x)
	  	\end{pmatrix}
	  	\begin{pmatrix}
	  		a^1 \\
	  		\vdots \\
	  		a^k 
	  	\end{pmatrix} 
	  	= \sum_{\alpha, \beta = 1}^k a^\alpha a^\beta \frac{\del \left(m^C\right)^\alpha}{\del x^\beta} (x)\\
	  	&= \frac{-1}{(\sum_j e^{-\langle c_j | x \rangle})^2} \sum_{\alpha, \beta} a^\alpha a^\beta \left [ \sum_{i <j} e^{-\langle c_i + c_j | x \rangle} \left(c_i^\alpha - c_j^\alpha\right)\left(c_i^\beta - c_j^\beta\right) \right] \\
	  	&= \frac{-1}{(\sum_j e^{-\langle c_j | x \rangle})^2} \sum_{i <j} e^{-\langle c_i + c_j | x \rangle} \left[ \sum_{\alpha, \beta} a^\alpha  \left(c_i^\alpha - c_j^\alpha\right) a^\beta \left(c_i^\beta - c_j^\beta\right) \right] \\
	  	&= \frac{-1}{(\sum_j e^{-\langle c_j | x \rangle})^2} \sum_{i <j} e^{-\langle c_i + c_j | x \rangle} \left (\sum_{\alpha} a^\alpha \left(c_i^\alpha - c_j^\alpha\right) \right)^2 \\
	  	&= \frac{-1}{(\sum_j e^{-\langle c_j | x \rangle})^2} \sum_{i <j} e^{-\langle c_i + c_j | x \rangle} \left ( \langle a | c_i - c_j \rangle \right)^2 < 0. 
	  \end{align*}
	  
	  As the vector \(a\) was arbitrary, negative definiteness is shown. 
\end{proof}

\begin{thm} \label{thm:mc}
	The map $m^C$ defined above is bijective. 
\end{thm}

\begin{proof}
	 To show injectivity, define the function 
	  \begin{align*}
	        f: \RR^k &\lora \RR, \\
	        x &\longmapsto -\ln\left(\sum_{i = 1}^r e^{-\langle c_i | x \rangle}\right). 
	  \end{align*}

	  Then  $m^C= \nabla f$ is the gradient of $f$. By Lemma \ref{lem:derivative_mC_neg_def}, the Hessian \(\nabla^2 f = \nabla m^C\) of \(f\) is negative definite and therefore $f$ is strictly concave by the second-order condition for concavity given in Lemma \ref{lem:convex_properties}.  
	  
	  Let \(x,y \in \RR^k\) be given with \(x \neq y\). We use the first-order condition for concavity (Lemma \ref{lem:convex_properties}) twice to get
	  \begin{align*}
	  	f(x) &< f(y) + \nabla f(y)^T  (x - y)  \\
	  	&< f(x) +  \nabla f(x)^T  (y - x)  +  \nabla f(y)^T (x - y). 
	  \end{align*}
	  Consequently, as $\nabla f(x) = m^C(x)$, we have $\left(m^C(x) - m^C(y)\right)^T (y - x) > 0$. Therefore $m^C(x) \neq m^C(y)$ for all $x \neq y \in \RR^k$ and  injectivity is shown. \\
	  
	  Let us now show surjectivity. As the derivative of $m^C$ at every point is a negative definite matrix (Lemma \ref{lem:derivative_mC_neg_def}), we know by the Inverse Function Theorem that $m^C$ is a local diffeomorphism and that its image is open in $\inte(C)$. It remains to show that the image is also closed. Assume that the image is open but not closed and take a point on the boundary of the image which lies in the interior of $C$, say $y \in \del m^C(\RR^k) \cap \inte(C)$. Thus, we can find a sequence $(y_n)_{n \in \NN} \subset m^C(\RR^k)$ converging to $y$. Let $(x_n)_{n \in \NN} \subset \RR^k$ be the sequence of preimages with \(m^C(x_n) = y_n\) for all \(n \in \NN\). Now there are two cases to distinguish. 
	
    If the sequence diverges, $x_n \lora \infty$, we can find a subsequence, also denoted by $(x_n)_{n \in \NN}$, which fulfills all properties of Lemma \ref{lem:subsequence} with respect to $C^\circ$. As we are only interested in limits, we assume by Lemma \ref{lem:mcshift} that $C$ contains the origin in its interior. Let $F \subset C^\circ$ be the face determined by the converging condition in Lemma \ref{lem:subsequence} and $E \subset C$ its dual. By Lemma \ref{lem:pairinginfty} and the third property of Lemma \ref{lem:subsequence}, we see that for an arbitrary vertex $c_E$ of $E$ we get
	
	\begin{align} \label{eq:SEremains_1}
		\pairr{c_E - c_j}{x_{n,F}} &\lora 
			\begin{cases}
				 0 &\text{ if } c_j \in E\\
				-\infty &\text{ if } c_j \notin E 
			\end{cases}
	\end{align}
	
	and 
	
	\begin{align*}
	        \pairr{c_E - c_j}{x_n^F} &\text{ is bounded for all \(j\).}
	\end{align*}
	
	Only those summands of  
	
	\begin{align} \label{eq:SEremains_2}
		m^C(x_n) = \sum_{i = 1}^r \frac{e^{-\langle c_i | x_n \rangle}}{\sum_{k =1 }^r e^{-\langle c_k  | x_n \rangle}} c_i = \sum_{i = 1}^r \frac{e^{\langle c_E - c_i  | x_n \rangle}}{\sum_{k =1 }^r e^{\langle c_E - c_k  | x_n \rangle}} c_i
	\end{align}
	
	 with $c_i \in E$ remain in the limit. Consequently, the limit $\lim_{n \ra \infty} m^C(x_n)$ of the subsequence is a convex combination of all those vertices spanning the face $E$ and thus lies in $E$. Since $E$ is in the boundary of $C$, this is a contradiction to the assumption that the image is open but not closed. 
	 
    It remains to consider the case where $(x_n)_n$ is bounded. Here again we  find a subsequence $(x_{n_k})_k$ converging to some point $x \in \RR^k$. By continuity of $m^C$ and uniqueness of limits we conclude that $y = m^C(x)$ lies in the image of $m^C$. As $y$ was some arbitrary boundary point, the image of $m^C$ is also closed in $C$. 
 
    All together $m^C(\RR^k) = C$ and the map \(m^C\) is surjective.
\end{proof}

\subsection{The actual proof of Theorem \ref{thm:homeo}}

Before we prove Theorem \ref{thm:homeo} we state a topological lemma that significantly simplifies the proof of continuity\footnote{We thank the unknown referee for the idea to include this lemma and its proof.}. 

\begin{lem} \label{lem:topological_convergence}
	Let \(X\) and \(Y\) be first countable Hausdorff topological spaces and let \(A \subset X\) be an open and dense subset. Let \(f: X \lora Y\) be a map satisfying the following two properties:
	\begin{itemize}
	 \item The restriction \(f:A \lora Y\) is continuous.
	 \item For every point \(x \in X \setminus A\) and every sequence \((a_n)_{n \in \NN} \subset A\) converging to \(x\) we have \(\lim_{n \ra \infty} f(a_n) = f(x)\).
	\end{itemize}
	Then \(f\) is continuous.
\end{lem}

\begin{proof}
	Given a point \(x \in X \) and a sequence \((x_n)_{n \in \NN} \subset X\) converging to \(x\) we have to show that \(\lim_{n \ra \infty} f(x_n) = f(x)\). We already know that this holds true for points and sequences in \(A\) as well as for points \(x \in X \setminus A\) in the complement and sequences in \(A\) converging to \(x\). Thus, it suffices to show convergence for points and sequences in the complement, i.e. for any point \(x \in X \setminus A\) and a sequence \((x_n)_{n \in \NN} \subset X \setminus A\) converging to \(x\).
	
	We will prove this by contradiction. To do so suppose we had a point \(x \in X \setminus A\) and a sequence \((x_n)_{n} \subset X \setminus A\) converging to \(x\), for which we find a subsequence \((x_n)_n \subset X \setminus A\) and an open neighborhood \(U \subset Y \) of \(y \coloneqq f(x)\) such that the sequence \((y_n)_n \coloneqq (f(x_n))_n\) lies completely outside of \(U\).

	As \(A \subset X\) is dense there is a sequence \((a_k^n)_{k \in \NN} \subset A\) converging to \(x_n\) for every \(x_n \in X \setminus A\), hence we also have \(f(a_k^n) \lora f(x_n) = y_n\) for \(k \ra \infty\) by the second assumption of this lemma. As \(y_n \notin U\) for each \(n \in \NN\) and since \(U\) is open, we get \(f(a_k^n) \notin U\) for all but finitely many \(k\)'s. Removing the corresponding finitely many points from the sequences \((a_k^n)_k\) for each \(n\) yield subsequences \((a_k^n)_k \subset A\) with \(f(a_k^n) \notin U\) for all \(n, k \in \NN\). By a diagonal argument we get a sequence \((a_m)_{m \in \NN} \subset A\) consisting of some members of the sequences \((a_k^n)_k\) (for \(n \in \NN\)), and converging to \(x\). By the second assumption of this lemma, we conclude \(f(a_m) \lora f(x) = y \in U\) for \(m \ra \infty\), but this contradicts the fact that \(f(a_m) \notin U\) for all \(m \in \NN\).
\end{proof}

Let us now state and prove the main theorem of this part of the paper, the existence of a homeomorphism between \(\Bo\) and \(\xhor\).

\begin{thm}\label{thm:homeo}
	Let $(X, \normm)$ be a finite-dimensional normed space with a polyhedral norm. Let $B  \subset X$ be the unit ball associated to $\normm$ and $\Bo \subset X^*$ its dual. Then the horofunction compactification $\overline{X}^{hor}$ is homeomorphic to $\Bo$ via the map
	\begin{align*}
		m: \xhor \lora \Bo, \hspace{0.8cm} \begin{cases} \hfill	X \ni x  &\longmapsto m^{\Bo}(x), \\
			\delhor X \ni h_{E,p}   &\longmapsto m^E(p). \end{cases} 
	\end{align*}	
\end{thm}

\begin{proof}
      The proof is structured as follows: After showing that the map \(m\) is well-defined and bijective we prove continuity. As both spaces involved are Hausdorff and compact, this is enough to conclude that the map is a homeomorphism.

      Let $E$ be a face of $\Bo$ and $F = E^\circ \subset B$ its dual face. The map \(m^C\), as defined above, maps \(\RR^k\) into the interior of the \(k\)-dimensional convex set \(C \subset \left(\RR^k\right)^*\). Therefore we have \(m^{\Bo}(X) \subset \inte(\Bo)\) and the mapping is continuous and bijective. 
      
      Let us now look at \(m^E\) for a proper face \(E \subset \Bo\). As \(E\) lives in a proper affine subspace of \(X^*\) and is not of full dimension, we have to show that \(m^E(p) \in E\) with \(p \in V(F)^\bot\). Let \(\Bo = \conv\{u_1, \ldots, u_r\}\) and denote by $E^{F^*} = \conv\left\{u_j^{F^*} | u_j \in E \right\}$ the orthogonal projection of $E$ to $\left(V(F)^\bot\right)^* \subset X^*$, where \(\left(V(F)^\bot\right)^*\) denotes the subspace in \(X^*\) dual to \(V(F)^\bot \subset X\) and the upper script \(^{F^*}\) denotes the orthogonal projection to it. Note that \(\left(V(F)^\bot\right)^*\) is the subspace of \(X^*\) which is parallel to \(\aff(E)\) and has the same dimension as \(\aff(E)\).  By the construction of the dual unit ball, $E^{F^*}$ is a top-dimensional convex polytope in the vector space $\left(\Fbot\right)^*$.
      
      Let us briefly only consider the vector space \(V(F)^\bot\) and its dual \(\left(V(F)^\bot\right)^*\). Since \(p \in V(F)^\bot\)  by the definition of horofunctions we have \(m^{E^{F^*}} \in \inte\left(E^{F^*}\right)\) and, by Theorem \ref{thm:mc}, \(m^{E^{F^*}}\) is a continuous and bijective map from $\Fbot$ to $\inte\left(E^{F^*}\right)$. 
      
      Back to  \(X^*\). We know by Lemma \ref{lem:samepairing} that \(E^{F^*}\) is a parallel translate of \(E\) in \(X^*\). Thus, there is a $t \in V(F)^* \subset X^*$ such that $E = E^{F^*} + t$ is a top-dimensional convex polytope in the affine space $\left(V(F)^\bot\right)^* + t = \aff(E)$. To see that \(m^E\) is onto let $y \in E$ be a point,  $y^{F^*} \in E^{F^*}$ its orthogonal projection to \(\left(V(F)^\bot\right)^*\). Let $x \in \Fbot$ be the preimage of $y^{F^*}$. Then      
      \begin{align*}
        m^E(x) = m^{E^{F^*}}(x) + t = y^{F^*} + t = y.  
      \end{align*}
      By Lemma \ref{lem:mcshift} we conclude that the map $m^E$ maps $V(F)^\bot$ into $\inte(E)$ and that it is also continuous and bijective. This concludes the proof of well-definedness and bijectivity of \(m\). 
	
      To prove continuity note that $\Bo$ is the finite union of the relative interiors of its faces and that the map \(m\) is composed of the maps \(m^E\) for exactly the same faces \(E\) of \(\Bo\). We already know that each map \(m^E\) is continuous. Thus it remains to show that the composition of \(m\) is compatible with continuity.
      
      Let $(z_n)_{n \in \NN} \subset X$ be a sequence that converges to a horofunction $h_{E,p}$. Then, by the third property of the characterization of sequences in Theorem \ref{thm:characterization}, we know that $z_n^F$ converges to $p \in V(F)^\bot$. 
      By the same calculation as already done in Equations (\ref{eq:SEremains_1}) and (\ref{eq:SEremains_2}) on page \pageref{eq:SEremains_1} we conclude that, for $n \lora \infty$,
      \begin{align*} \label{eq:mu(z_n)}
			m(z_n) \lora \frac{\sum_{u_j \in E} e^{-\langle u_j | p \rangle} u_j}{\sum_{u_k \in E} e^{- \langle u_k | p \rangle}} = m(h_{E,p}).
      \end{align*}      
      In other words, we have just shown that for a horofunction \( h_{E,p} \in \xhor\) and a sequence \((z_n)_{n \in \NN} \subset X\) converging to \(h_{E,p}\), we have
      \[
		\lim_{n \ra \infty} m(z_n) = m(h_{E,p}).
      \]
      As \(X \subset \xhor\) is a dense subset we conclude by Lemma \ref{lem:topological_convergence}  that our map \(m: \xhor \lora \Bo\) is continuous. \qedhere

\end{proof}

\bibliography{references}{}

\begin{thebibliography}{HSWW17}

\bibitem[ALPS16]{alessandrini}
Daniele Alessandrini, Lixin Liu, Athanase Papadopoulos, and Weixu Su.
\newblock The horofunction compactification of the arc metric on
  teichm{\"u}ller space, 2016.

\bibitem[Bee93]{be}
Gerald Beer.
\newblock {\em Topologies on closed and closed convex sets}, volume 268 of {\em
  Mathematics and its Applications}.
\newblock Kluwer Academic Publishers Group, Dordrecht, 1993.

\bibitem[BV04]{BV}
Stephen Boyd and Lieven Vandenberghe.
\newblock {\em Convex optimization}.
\newblock Cambridge University Press, Cambridge, 2004.

\bibitem[CCAL22]{lemmens}
Cho-Ho Chu, María Cueto-Avellaneda, and Bas Lemmens.
\newblock Horofunctions and metric compactification of noncompact hermitian
  symmetric spaces, 2022.

\bibitem[CKS23]{cks}
Corina Ciobotaru, Linus Kramer, and Petra Schwer.
\newblock Polyhedral compactifications, i, 2023.

\bibitem[Ful93]{fu}
William Fulton.
\newblock {\em Introduction to toric varieties}, volume 131 of {\em Annals of
  Mathematics Studies}.
\newblock Princeton University Press, Princeton, NJ, 1993.
\newblock The William H. Roever Lectures in Geometry.

\bibitem[Gro81]{gr}
M.~Gromov.
\newblock Hyperbolic manifolds, groups and actions.
\newblock In {\em Riemann surfaces and related topics: {P}roceedings of the
  1978 {S}tony {B}rook {C}onference ({S}tate {U}niv. {N}ew {Y}ork, {S}tony
  {B}rook, {N}.{Y}., 1978)}, volume~97 of {\em Ann. of Math. Stud.}, pages
  183--213. Princeton Univ. Press, Princeton, N.J., 1981.

\bibitem[HSWW17]{hsw}
Thomas Haettel, Anna-Sofie Schilling, Cormac Walsh, and Anna Wienhard.
\newblock Horofunction compactifications of symmetric spaces, 2017.

\bibitem[Ji97]{ji}
Lizhen Ji.
\newblock Satake and {M}artin compactifications of symmetric spaces are
  topological balls.
\newblock {\em Math. Res. Lett.}, 4(1):79--89, 1997.

\bibitem[JS17]{js}
Lizhen Ji and Anna-Sofie Schilling.
\newblock Toric varieties vs. horofunction compactifications of polyhedral
  norms.
\newblock {\em Enseign. Math.}, 63(3-4):375--401, 2017.

\bibitem[KL18]{kl}
Michael Kapovich and Bernhard Leeb.
\newblock Finsler bordifications of symmetric and certain locally symmetric
  spaces.
\newblock {\em Geom. Topol.}, 22(5):2533--2646, 2018.

\bibitem[KMN06]{KMN}
A.~Karlsson, V.~Metz, and G.~A. Noskov.
\newblock Horoballs in simplices and {M}inkowski spaces.
\newblock {\em Int. J. Math. Math. Sci.}, pages Art. ID 23656, 20, 2006.

\bibitem[Rie02]{ri}
Marc~A. Rieffel.
\newblock Group {$C^*$}-algebras as compact quantum metric spaces.
\newblock {\em Doc. Math.}, 7:605--651, 2002.

\bibitem[Roc97]{ro}
R.~Tyrrell Rockafellar.
\newblock {\em Convex analysis}.
\newblock Princeton Landmarks in Mathematics. Princeton University Press,
  Princeton, NJ, 1997.
\newblock Reprint of the 1970 original, Princeton Paperbacks.

\bibitem[Sch14]{sch14}
Anna-Sofie Schilling.
\newblock Horofunction compactification of finite-dimensional normed spaces and
  of symmetric spaces.
\newblock Master's thesis, Universit{\"a}t Heidelberg, Germany, 2014.
\newblock Diplomarbeit.

\bibitem[Sch21]{sch21}
Anna-Sofie Schilling.
\newblock {\em The Horofunction Compactification of Finite-Dimensional Normed
  Vector Spaces and of Symmetric Spaces}.
\newblock PhD thesis, Universit{\"a}t Heidelberg, Germany, 2021.

\bibitem[Wal07]{wa2}
Cormac Walsh.
\newblock The horofunction boundary of finite-dimensional normed spaces.
\newblock {\em Math. Proc. Cambridge Philos. Soc.}, 142(3):497--507, 2007.

\bibitem[Wal14]{wa1}
Cormac Walsh.
\newblock The horoboundary and isometry group of {T}hurston's {L}ipschitz
  metric.
\newblock In {\em Handbook of {T}eichm{\"u}ller theory. {V}ol. {IV}}, volume~19
  of {\em IRMA Lect. Math. Theor. Phys.}, pages 327--353. Eur. Math. Soc.,
  Z{\"u}rich, 2014.

\end{thebibliography}
\bibliographystyle{alpha}

\end{document}